\documentclass[
final
]{dmtcs-episciences}

\usepackage[utf8]{inputenc}
\usepackage{subfigure}

\usepackage{amsmath}

\def\pp{{\mathcal P}}
\def\ii{{\mathcal I}}
\def\jj{{\mathcal J}}

\def\w{{M}}
\def\wb{{B}}
\def\wa{{A}}
\def\PP{{\mathbb P}}
\def\QQ{{\mathbb Q}}

\def\TT{{\mathbb T}}
\def\ZZ{{\mathbb Z}}

\newtheorem{theorem}{Theorem}

\newtheorem{corollary}{Corollary}

\newtheorem{lemma}{Lemma}

\newtheorem{proposition}{Proposition}
\newtheorem{remark}{Remark}
\newtheorem{observation}{Observation}

\usepackage[round]{natbib}

\author{J. Lea\~nos\affiliationmark{1}
  \and Christophe Ndjatchi\affiliationmark{2}
  \and L. M. R\'ios-Castro\affiliationmark{3}}
\title[On the connectivity of the disjointness graph of segments]{On the connectivity of the disjointness graph of segments of point sets in general position in the plane}
\affiliation{
  Unidad Acad\'emica de Matem\'aticas, Universidad Aut\'onoma de Zacatecas, M\'exico.\\
  Academia de F\'isico-Matem\'aticas, Instituto Polit\'ecnico Nacional, UPIIZ, Zacatecas, M\'exico.\\
  Academia de F\'isico-Matem\'aticas, Instituto Polit\'ecnico Nacional, CECYT 18, Zacatecas, M\'exico.
  }
\received{2020-07-31}
\revised{2021-08-25, 2021-12-13, 2022-04-06}
\accepted{2022-04-07}

\begin{document}
\publicationdetails{24}{2022}{1}{15}{6678}

\maketitle

\begin{abstract}
 Let \begin{math}P\end{math} be a set of  \begin{math}n\geq 3\end{math} points in general position in the plane. The \emph{edge disjointness graph} \begin{math}D(P)\end{math} of \begin{math}P\end{math} is the graph whose vertices are all the closed straight  line segments with endpoints in
\begin{math}P\end{math}, two of which are adjacent in \begin{math}D(P)\end{math} if and only if they are disjoint. We show that the connectivity of \begin{math}D(P)\end{math} is at least
  \begin{math}
\binom{\lfloor\frac{n-2}{2}\rfloor}{2}+\binom{\lceil\frac{n-2}{2}\rceil}{2}
\end{math},
and that this bound is tight for each \begin{math}n\geq 3\end{math}.
\end{abstract}

\section{Introduction}
\label{sec:in}
Throughout this paper, \begin{math}P\end{math} is a set of  \begin{math}n\geq 3\end{math} points in general position in the plane, i.e., no three points in \begin{math}P\end{math} are collinear. The \emph{edge disjointness graph} \begin{math}D(P)\end{math} of \begin{math}P\end{math} is the graph whose vertices correspond to the
closed straight  line segments with endpoints in \begin{math}P\end{math} and in which two vertices are adjacent if and only if the corresponding segments are disjoint. Figure~\ref{fig:Ejem} depicts a point set \begin{math}P\end{math}, \begin{math}\mathcal{P}\end{math} and \begin{math}D(P)\end{math}.

The edge disjointness graph and  other similar graphs were introduced by Araujo, Dumitrescu, Hurtado, Noy and Urrutia in\cite{gaby}, as geometric analogs of the
  Kneser graphs. We recall that if \begin{math}m\end{math} and \begin{math}k\end{math} are positive integers with \begin{math}k\leq  m/2\end{math}, then the \emph{Kneser graph} \begin{math}KG(m; k)\end{math} is the graph whose vertices are all the \begin{math}k\end{math}--subsets of
\begin{math}\{1,2,\ldots ,m\}\end{math} and in which two vertices are adjacent if and only if the corresponding \begin{math}k\end{math}-subsets are disjoint. Kneser conjectured~\cite{kneser} in 1956 that the chromatic number \begin{math}\chi(KG(m; k))\end{math} of \begin{math}KG(m; k)\end{math}
is equal to \begin{math} m-2k+2\end{math}. This conjecture was proved by Lov\'asz~\cite{lovasz} in 1978 using topological methods, and (independently) by B\'ar\'any~\cite{barany} in the same year. For more about the
 study of several other aspects and combinatorial properties of Kneser graphs, see for instance~\cite{albertson,matousek,chen,baptist,valencia}.

The study of the graph invariants of the edge disjointness graph \begin{math}D(P)\end{math} began in~\cite{gaby} with the estimation of a general lower bound for the chromatic number \begin{math}\chi(D(P))\end{math} of \begin{math}D(P)\end{math}.
 Up to now the problem of determining the exact value of \begin{math}\chi(D(P))\end{math} remains open in general. As far as we know, the exact
 value of \begin{math}\chi(D(P))\end{math} is known only for two particular cases: when \begin{math}P\end{math}
 is in convex position~\cite{ruy-wood,jonsson}, and when \begin{math}P\end{math} is the double chain~\cite{lomeli}. In 2017 Pach, Tardos, and T\'oth~\cite{pach-tardos-toth}
 studied the chromatic number and the clique number of \begin{math}D(P)\end{math} in the more general setting of \begin{math}\reals^d\end{math} for \begin{math}d\geq 2\end{math}, i.e., when \begin{math}P\end{math} is a subset of \begin{math}\reals^d \end{math}.
 More precisely, in~\cite{pach-tardos-toth} it was shown that the chromatic number of \begin{math}D(P)\end{math} is bounded by above by a polynomial function that depends on its clique number \begin{math}\omega(D(P))\end{math},
  and that the problem of determining any of \begin{math}\chi(D(P))\end{math} or \begin{math}\omega(D(P))\end{math} is NP-hard. Two years later, Pach and Tomon~\cite{pach-tomon}
 showed that if \begin{math}G\end{math} is the disjointness graph of a set of grounded \begin{math}x\end{math}-monotone curves in \begin{math}\reals^2\end{math} and \begin{math}\omega(G)=k\end{math}, then \begin{math}\chi(G)\leq k+1\end{math}. We remark that the set of grounded \begin{math}x\end{math}-monotone curves play the role of our closed straight line segments.

The basic notations that we will use in this work are the following. If \begin{math}x\end{math} and \begin{math}y\end{math} are distinct points of \begin{math}P\end{math}, then we shall use \begin{math}xy\end{math} to
denote the closed straight  line segment whose endpoints are \begin{math}x\end{math} and \begin{math}y\end{math}. Similarly, we will use \begin{math}\pp\end{math} to denote the set of segments \begin{math}\{xy~:~x,y\in P \mbox{ and } x\neq y\}\end{math},
and we shall refer to the elements of \begin{math}\pp\end{math} as the \emph{segments} of \begin{math}\pp\end{math}. Then, \begin{math}\pp\end{math} is the vertex set of \begin{math}D(P)\end{math}. We often make no distinction between
 an element of \begin{math}\pp\end{math} and its corresponding vertex in \begin{math}D(P)\end{math}. We also note that \begin{math}\pp\end{math} naturally defines a rectilinear drawing of \begin{math}K_n\end{math} in the plane.
 Let \begin{math}x_1y_1\end{math} and \begin{math}x_2y_2\end{math} be two distinct elements of \begin{math}\pp\end{math}, and suppose that \begin{math}x_1y_1\cap x_2y_2\neq \emptyset\end{math}. Then \begin{math}x_1y_1\cap x_2y_2\end{math}
 consists precisely of one point \begin{math}o\in \reals^2\end{math}, because \begin{math}P\end{math} is in general position. If \begin{math}o\end{math} is an interior point
  of both \begin{math}x_1y_1\end{math} and \begin{math}x_2y_2\end{math}, then we say that they \emph{cross} at \begin{math}o\end{math}.

  We will denote by \begin{math}CH(P)\end{math} the boundary of the convex hull of \begin{math}P\end{math}, and by \begin{math}\overline{P}\end{math} the set \begin{math}P\cap CH(P)\end{math}, as depicted in Figure~\ref{fig:Ejem}.
In particular, note that if \begin{math}P\end{math} is in convex position, then \begin{math}P=\overline{P}\end{math}.
 \begin{figure}
	\centering
	\includegraphics[width=0.9\textwidth]{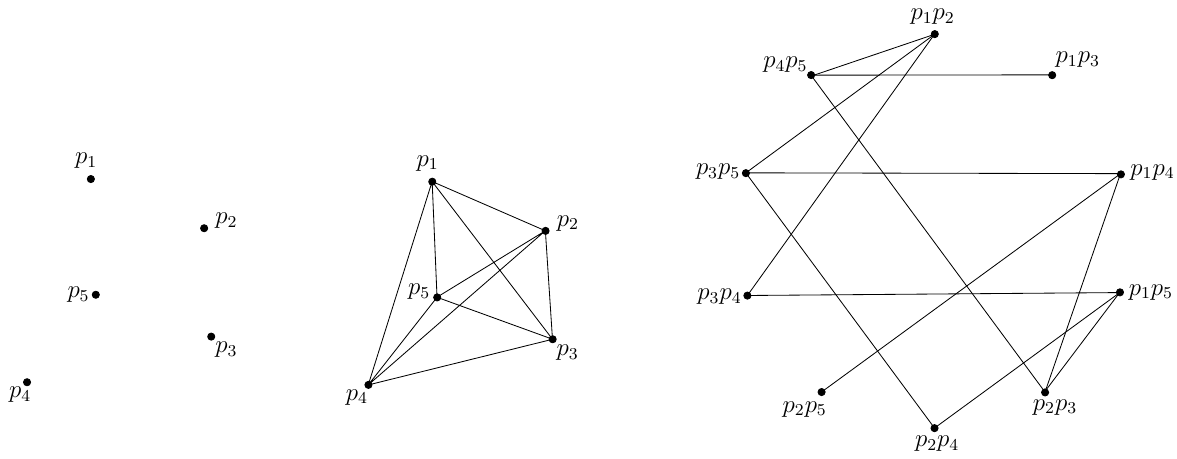}
	\caption{\small The set \begin{math}\{p_1,\ldots ,p_5\}\end{math} of points in general position on the left is \begin{math}P\end{math}. We note that \begin{math}\overline{P}=\{p_1, p_2, p_3, p_4\}\end{math}. In the middle we have \begin{math}\pp\end{math}, which can be seen as the rectilinear drawing of \begin{math}K_5\end{math} induced by \begin{math}P\end{math}. Note that \begin{math}CH(P)\end{math} is the convex quadrilateral formed by the union of the segments \begin{math}p_1p_2, \: p_2p_3, \:p_3p_4,\end{math} and \begin{math}p_4p_1\end{math}. The graph on the right is the edge disjointness graph \begin{math}D(P)\end{math}
	corresponding to \begin{math}P\end{math}.}
		\label{fig:Ejem}
\end{figure}
Let \begin{math}H=(V(H),E(H))\end{math} be a (non-empty) simple connected graph. If \begin{math}u\end{math} and \begin{math}v\end{math} are distinct vertices of \begin{math}H\end{math},
then the distance between \begin{math}u\end{math} and \begin{math}v\end{math} in \begin{math}H\end{math} will be denoted by \begin{math}d_H(u,v)\end{math}, and we write \begin{math}uv\end{math} to mean that \begin{math}u\end{math} and \begin{math}v\end{math} are adjacent in \begin{math}H\end{math}.
We  emphasize that this last notation is similar to that used to denote the straight line segment \begin{math}xy\end{math} defined by the points \begin{math}x,y\in \reals^2\end{math}.
However, none of these notations should be a source of confusion, because the former objects are vertices of a graph, and the
latter are points of the plane.

The \emph{neighborhood} of \begin{math}v\end{math} in \begin{math}H\end{math} is the set \begin{math}\{u\in V(H): uv\in E(H)\}\end{math} and is denoted by \begin{math}N_H(v)\end{math}. The \emph{degree} \begin{math}\deg_H(v)\end{math} of \begin{math}v\end{math} is the number \begin{math}|N_H(v)|\end{math}.
The number  \begin{math}\delta(H):=\min\{\deg_H(v): v\in V(H)\}\end{math} is the \emph{ minimum degree} of \begin{math}H\end{math}, and  \begin{math}\Delta(H):=\max\{\deg_H(v): v\in V(H)\}\end{math} is its \emph{ maximum degree}.
A \emph{\begin{math}u-v\end{math} path of} \begin{math}H\end{math} is a path of \begin{math}H\end{math} having an endpoint in \begin{math}u\end{math} and the other endpoint in \begin{math}v\end{math}. Similarly, if \begin{math}U\subset V(H)\end{math}, then
\begin{math}H\setminus U\end{math} is the subgraph of \begin{math}H\end{math} that results by removing the vertices of \begin{math}U\end{math} from \begin{math}H\end{math}.

We recall that if \begin{math}k\end{math} is a nonnegative integer, then \begin{math}H\end{math} is \emph{\begin{math}k\end{math}--connected} if \begin{math}|V(H)|>k\end{math} and \begin{math}H\setminus W\end{math} is connected for every set \begin{math}W\subset V(H)\end{math} with \begin{math}|W|<k\end{math}. The \emph{connectivity} \begin{math}\kappa(H)\end{math} of \begin{math}H\end{math} is the greatest integer \begin{math}k\end{math} such that \begin{math}H\end{math} is \begin{math}k\end{math}-connected.
We follow the usual convention that \begin{math}\kappa(H)=0\end{math} if and only if \begin{math}H\end{math} is disconnected or \begin{math}|V(H)|=1\end{math}.

Throughout this paper, if \begin{math}m\end{math} is a nonnegative integer, then \begin{math}[m]:=\emptyset\end{math} if \begin{math}m=0\end{math}, \begin{math}[m]:=\{1,\ldots,m\}\end{math} if \begin{math}m>0\end{math}, and
 by convention \begin{math}\binom{m}{2}:=0\end{math} if \begin{math}m<2\end{math}.

Our aim in this paper is to study the connectivity of \begin{math}D(P)\end{math}. As far as we know, this parameter of \begin{math}D(P)\end{math} has not been investigated previously, which is somehow surprising for us. Our main result is the following.
\begin{theorem}\label{thm:main} If \begin{math}P\end{math} is any set of \begin{math}n\geq 3\end{math} points in general position in the plane, then
\begin{displaymath}\kappa(D(P))\geq \binom{\lfloor\frac{n-2}{2}\rfloor}{2}+\binom{\lceil\frac{n-2}{2}\rceil}{2}.\end{displaymath}
\end{theorem}

Our next statement is an immediate consequence of Theorem~\ref{thm:main} and the well-known fact that the connectivity of
a graph is bounded by above by  its minimum degree.
\begin{corollary}\label{cor:main} If \begin{math}P\end{math} is any set of \begin{math}n\geq 3\end{math} points in general position in the plane and \begin{math}\delta(D(P))=\binom{\lfloor\frac{n-2}{2}\rfloor}{2}+\binom{\lceil\frac{n-2}{2}\rceil}{2},\end{math}
then  \begin{displaymath}\kappa(D(P))= \binom{\lfloor\frac{n-2}{2}\rfloor}{2}+\binom{\lceil\frac{n-2}{2}\rceil}{2}.\end{displaymath}
\end{corollary}

As we will see in Proposition~\ref{prop:easy}, the lower bound of \begin{math}\kappa(D(P))\end{math} given in Theorem~\ref{thm:main} is in fact a lower bound for the minimum degree of \begin{math}D(P)\end{math}.
Moreover, in that proposition we also show that each point set in the collection \begin{math}\{C_n\}_{n=3}^{\infty}\end{math}, where \begin{math}C_n\end{math} denotes the set of \begin{math}n\end{math} points in
general and convex position, satisfy the hypothesis of Corollary~\ref{cor:main}, and hence \begin{math}\kappa(D(C_n))\end{math} is equal to the lower bound in Theorem~\ref{thm:main} for each \begin{math}n\geq 3\end{math}.

The rest of the paper is organized as follows. In Section~\ref{sec:preliminaries} we introduce additional  terminology and give some auxiliary results which will be used in the proof Theorem~\ref{thm:main}. Finally, in Section~\ref{sec:main} we give the proof of our main result.


\section{Preliminaries}\label{sec:preliminaries}
For the rest of the paper, \begin{math}P\end{math} is a fixed set of \begin{math}n\geq 3\end{math} points in general position in the plane, and (for brevity)
 \begin{math}\kappa(n):=\binom{\lfloor\frac{n-2}{2}\rfloor}{2}+\binom{\lceil\frac{n-2}{2}\rceil}{2}.\end{math}
  Similarly, we will use \begin{math}\eta(P;a,b)\end{math} to denote the maximum number of pairwise internally
 disjoint \begin{math}a-b\end{math} paths of \begin{math}D(P)\end{math}.

\begin{proposition}\label{prop:easy}
If \begin{math}P\end{math} and \begin{math}n\end{math} are as above, then \begin{math}\delta(D(P))\geq \kappa(n)\end{math} and  \begin{math}\Delta(D(P))=\binom{n-2}{2}\end{math}.
If additionally \begin{math}P\end{math} is in convex position, then \begin{math}\delta(D(P))= \kappa(n)\end{math}.
\end{proposition}
\begin{proof} Let \begin{math}f=uv\end{math} be a vertex of \begin{math}D(P)\end{math}, and let \begin{math}P_1\end{math} and \begin{math}P_2\end{math} be the subsets of \begin{math}P\setminus \{u,v\}\end{math} separated by the line spanned by \begin{math}f\end{math}.
 Let \begin{math}n_1:=|P_1|\end{math} and \begin{math}n_2:=|P_2|\end{math}. Then, \begin{math}n_1+n_2=n-2\end{math}.
Since each segment of \begin{math}\pp\end{math} having both endpoints in \begin{math}P_i\end{math} (\begin{math}i=1,2\end{math}) is disjoint from \begin{math}f\end{math}, then each of these segments is adjacent to \begin{math}f\end{math} in \begin{math}D(P)\end{math}. Hence \begin{math}\deg(f)\geq \binom{n_1}{2}+\binom{n_2}{2}\end{math}.
On the other hand, it is well known that the sum  \begin{math}\binom{n_1}{2}+\binom{n_2}{2}\end{math} attains its minimum when
\begin{math}\{n_1,n_2\}=\{\lfloor \frac{n-2}{2}\rfloor, \: \lceil\frac{n-2}{2}\rceil\}\end{math}. From the last two assertions it follows that
\begin{math}\deg(f)\geq \kappa(n)\end{math}. Since \begin{math}f\end{math} is an arbitrary vertex of \begin{math}D(P)\end{math}, then
\begin{math}\delta(D(P))\geq  \kappa(n)\end{math}.

It follows from  \begin{math}n\geq 3\end{math} and the fact that \begin{math}P\end{math} is in general position that \begin{math}CH(P)\end{math} is a polygon of at least three sides. Note that if \begin{math}g=xy\end{math} is a segment (side) of \begin{math}CH(P)\end{math}, then \begin{math}g\end{math} is disjoint  from
any segment joining two points of \begin{math}P\setminus \{x,y\}\end{math}. This implies that \begin{math}\deg(g)\geq \binom{n-2}{2}\end{math}, and so \begin{math}\Delta(D(P))\geq \binom{n-2}{2}\end{math}. On the other hand, note that
for \begin{math}f=uv\end{math} there are exactly \begin{math}2(n-2)\end{math} segments of \begin{math}\pp\setminus \{f\} \end{math} that share an endpoint with \begin{math}f\end{math}, (namely, those incident with exactly one of \begin{math}u\end{math} or \begin{math}v\end{math}). Since \begin{math}f\end{math} cannot be adjacent to any of these
segments, then \begin{math}\deg(f)\leq  \binom{n}{2}-1-2(n-2)=\binom{n-2}{2}\end{math}. Again, since \begin{math}f\end{math} is an arbitrary vertex of \begin{math}D(P)\end{math}, then \begin{math}\Delta(D(P))\leq \binom{n-2}{2}\end{math}, as required.

Finally, suppose that \begin{math}P\end{math} is in convex position. Let us label the points of \begin{math}P\end{math} by  \begin{math} x_1,\: x_2,\:\ldots ,\:x_n \end{math} in clockwise order. Let \begin{math}h\end{math} be the segment of \begin{math}\pp\end{math} joining \begin{math}x_1\end{math} with \begin{math}x_j\end{math}, where \begin{math}j:=\lfloor (n+2)/2\rfloor\end{math}.
Then the line spanned by \begin{math}h\end{math} separates \begin{math} S_1:=\{x_2,\ldots ,x_{j-1}\}\end{math} from \begin{math}S_2:=\{x_{j+1},\ldots ,x_n\}\end{math}. Then \begin{math}|S_1|=\lfloor\frac{n-2}{2}\rfloor\end{math} and \begin{math}|S_2|=\lceil\frac{n-2}{2}\rceil\end{math}. Since \begin{math}P\end{math} is in convex position, then any segment of \begin{math}\pp\end{math}
with an endpoint in \begin{math}S_1\end{math} and the other in \begin{math}S_2\end{math} crosses \begin{math}h\end{math}, and hence the only neighbours of \begin{math}h\end{math} in \begin{math}D(P)\end{math} are those segments that have both endpoints in exactly one of \begin{math}S_1\end{math} or \begin{math}S_2\end{math}. This implies that
\begin{math}\deg(h)\leq \kappa(n)\end{math}, showing the last assertion of Proposition~\ref{prop:easy}.
\end{proof}

\begin{proposition}
	\label{pro:restric}
	Let \begin{math}H\end{math} be a connected graph. Then \begin{math}H\end{math} is \begin{math}k\end{math}-connected if and only if \begin{math}H\end{math} has
	\begin{math}k\end{math} pairwise internally disjoint \begin{math}a-b\end{math} paths, for any two vertices \begin{math}a\end{math} and \begin{math}b\end{math} of \begin{math}H\end{math} such that \begin{math}d_H(a,b)=2\end{math}.
\end{proposition}
\begin{proof} The forward implication follows directly from Menger's Theorem. Conversely, let \begin{math}U\end{math} be a vertex cut of \begin{math}H\end{math} of minimum order.
Let \begin{math}H_1\end{math} and \begin{math}H_2\end{math} be two distinct components  of \begin{math}H\setminus U\end{math}, and let \begin{math}u\in U\end{math}. Since \begin{math}U\end{math} is a minimum cut, then \begin{math}u\end{math} has at least a neighbor
\begin{math}v_i\end{math} in \begin{math}H_i\end{math}, for \begin{math}i=1,2\end{math}. Then \begin{math}d_H(v_1,v_2)=2\end{math}. By hypothesis, \begin{math}H\end{math} has \begin{math}k\end{math} pairwise internally disjoint \begin{math}v_1-v_2\end{math} paths. Since each of these
\begin{math}k\end{math} paths intersects \begin{math}U\end{math}, then we have that \begin{math}|U|\geq k\end{math}, as required.
\end{proof}

\begin{remark}\label{lem:main} Let \begin{math}a,b\end{math} be vertices of \begin{math}D(P)\end{math} such that \begin{math}d_{D(P)}(a,b)=2\end{math}. By Proposition~\ref{pro:restric} and
Menger's Theorem, in order to show Theorem \ref{thm:main} it is enough to show that \begin{math}\eta(P;a,b)\geq \kappa(n)\end{math}.
\end{remark}

In view of Remark~\ref{lem:main}, for the rest of the paper we can assume that \begin{math}a\end{math} and \begin{math}b\end{math} are two fixed vertices of \begin{math}D(P)\end{math} such that \begin{math}d_{D(P)}(a,b)=2\end{math}.
Then \begin{math}a\end{math} and \begin{math}b\end{math} are not adjacent in \begin{math}D(P)\end{math}, and hence \begin{math}a\cap b\neq \emptyset\end{math}. This inequality and the fact that the points of \begin{math}P\end{math} are in general position
imply that \begin{math}a\cap b\end{math} consists precisely of one point of \begin{math}\reals^2\end{math}, which will be denoted by \begin{math}o\end{math}. Then  \begin{math}a\end{math} and \begin{math}b\end{math} cross at \begin{math}o\end{math}, or \begin{math}o\end{math} is
 common endpoint of \begin{math}a\end{math} and \begin{math}b\end{math}.

An endpoint of \begin{math}a\end{math} or \begin{math}b\end{math} that is in exactly one of \begin{math}a\end{math} or \begin{math}b\end{math} will be called a \emph{leaf} of \begin{math}\{a,b\}\end{math}. Thus, if \begin{math}a\end{math} and \begin{math}b\end{math} cross at \begin{math}o\end{math}, then each endpoint of
\begin{math}a\end{math} and \begin{math}b\end{math} is a leaf. Otherwise \begin{math}o\end{math} is a common endpoint of \begin{math}a\end{math} and \begin{math}b\end{math}, and each of \begin{math}a\end{math} and \begin{math}b\end{math} has exactly one leaf, namely the endpoint of \begin{math}a\end{math} (respectively, \begin{math}b\end{math}) distinct from \begin{math}o\end{math}. In particular, note that the number of leaves of \begin{math}\{a,b\}\end{math} is 2 or 4.

By translating \begin{math}P\end{math}, if necessary, from now on we will assume that \begin{math}o=(0,0)\end{math}. Let \begin{math}\ell_a\end{math} and \begin{math}\ell_b\end{math} be the straight lines spanned by \begin{math}a\end{math} and \begin{math}b\end{math}, respectively.  Additionally, by rotating \begin{math}P\end{math} around \begin{math}o=(0,0)\end{math}, if necessary, we also can assume that the slope of some of \begin{math}\ell_a\end{math} or \begin{math}\ell_b\end{math} is positive, and that the slope of the other one is negative. Clearly, \begin{math}\reals^2 \setminus\{\ell_a, \:\ell_b\}\end{math}
consists of four open connected regions \begin{math}R_{y^-}, \: R_{y^+}, \: R_{x^-},\end{math} and \begin{math}R_{x^+}\end{math}, where \begin{math}R_{y^-}\end{math} denotes the region containing the negative \begin{math}y\end{math}-axis, and so on. For brevity, in the rest of the paper, we use \begin{math}Y^-(P):=P\cap R_{y^-}, \: Y^+(P):=P\cap R_{y^+}, \: X^-(P):=P\cap R_{x^-}\end{math},  and \begin{math}X^+(P):=P\cap R_{x^+}\end{math}.  Then, \begin{math}P\end{math} without the endpoints of \begin{math}a\end{math} and \begin{math}b\end{math} is the disjoint union of \begin{math}Y^-(P), \: Y^+(P), \: X^-(P),\end{math} and \begin{math}X^+(P)\end{math}.  Let \begin{math}r:=|Y^-(P)|, \: s:=|Y^+(P)|, \: p-1:=|X^-(P)|,\end{math} and \begin{math}q-1:=|X^+(P)|\end{math}, as depicted in Figure \ref{fig:VI}. If there is no danger of confusion, we often omit the argument \begin{math}P\end{math} in all these expressions.

We recall that the set \begin{math}\pp\end{math} of segments with endpoints in \begin{math}P\end{math} is the vertex set of \begin{math}D(P)\end{math}.
We now split the set of neighbours of \begin{math}a\end{math} and \begin{math}b\end{math} into three sets as follows.
\begin{equation*}
\begin{split}
\wa&:=\{e\in \pp~|~ e\cap a=\emptyset \mbox{ and } e\cap b\neq \emptyset\},\\
\wb&:=\{e\in \pp~|~ e\cap b=\emptyset \mbox{ and } e\cap a\neq\emptyset\},\\
\w&:=\{e\in \pp~|~ e\cap a=\emptyset \mbox{ and } e\cap b=\emptyset\}.
\end{split}
\end{equation*}

Clearly, \begin{math} \wa, \: \wb, \: \w,\end{math} and \begin{math}\{a, \:b\}\end{math} are pairwise disjoint. Moreover, note that \begin{math} N_{D(P)}(a)=\wa \cup \w\end{math} and \begin{math}N_{D(P)}(b)=\wb \cup \w\end{math}.
Let \begin{math}\delta_2:=|\w|\end{math} and \begin{math}\delta_3:=\min\{|\wa|, |\wb|\}\end{math}. Let us denote by \begin{math}\delta(P;a,b)\end{math} the minimum of \begin{math}|A\cup M|\end{math} and \begin{math}|B\cup M|\end{math}.
Then, \begin{math}\delta(P;a,b)=\delta_2+\delta_3\end{math}.

Let \begin{math}G\end{math} be the subgraph of \begin{math}D(P)\end{math} induced by \begin{math}A \cup B \cup M \cup \{a, \: b\}\end{math}. In particular, note that
\begin{math}N_G(a)=N_{D(P)}(a)\end{math} and \begin{math}N_G(b)=N_{D(P)}(b)\end{math}. We shall see later that the subgraph \begin{math}G\end{math}
contains (almost all) the \begin{math}\kappa(n)\end{math} \begin{math}a-b\end{math} paths mentioned in Remark~\ref{lem:main}. The following  observation is easy to check.
\begin{observation}\label{claim:conditions}
Let \begin{math}d, \: h \in V(G)\setminus\{a, \: b\}\end{math}. Then \begin{math}adhb\end{math} is an \begin{math}a-b\end{math} path of \begin{math}G\end{math} of length 3 if and only if \begin{math}adhb\end{math} satisfies the following conditions:
(i) \begin{math}d \in A\end{math}, (ii) \begin{math}h \in B\end{math}, and (iii) \begin{math}d\end{math} and \begin{math}h\end{math} are disjoint.
\end{observation}
The next proposition provides a useful collection \begin{math}\PP_0\end{math} of pairwise internally disjoint \begin{math}a-b\end{math} paths of \begin{math}G\end{math}.
\begin{proposition}\label{cl:P0}
Let  \begin{math}a, \: b, M,\end{math} and \begin{math}G\end{math} be as above, and let \begin{math}\PP_0:=\{aeb ~|~ e\in \w\}.\end{math} If \begin{math}\eta_2(a,b)\end{math} denotes the maximum number of pairwise internally disjoint \begin{math}a-b\end{math} paths of \begin{math}G\end{math} of length \begin{math}2\end{math}, then \begin{math}\eta_2(a,b)=\delta_2=|\PP_0|\end{math}.
 \end{proposition}
\begin{proof} From the definitions of \begin{math}\delta_2\end{math} and \begin{math}\PP_0\end{math}, it is clear that \begin{math}\delta_2=|\PP_0|\end{math}. Similarly, from the definition of \begin{math}\w\end{math}
it follows that each element of \begin{math}\PP_0\end{math} is an \begin{math}a-b\end{math} path of \begin{math}G\end{math} of length 2. Conversely, if \begin{math}T\end{math} is an \begin{math}a-b\end{math} path of \begin{math}G\end{math} of length 2, then
the definition of \begin{math}G\end{math} implies that the inner vertex of \begin{math}T\end{math} must be a segment of \begin{math}\w\end{math}, and so \begin{math}\PP_0\end{math} consists precisely of all the \begin{math}a-b\end{math} paths of \begin{math}G\end{math} of length 2.
Since the paths in \begin{math}\PP_0\end{math} have length 2, then they are pairwise internally disjoint if and only if they are pairwise distinct.
Then \begin{math}\eta_2(a,b)=|\PP_0|\end{math}, as claimed.
\end{proof}

With the facts and the terminology given in this section in mind, we are ready to prove our main result.


\section{The proof of Theorem~\ref{thm:main}}\label{sec:main}
 We start by noting that \begin{math}\kappa(3)=0, \: \kappa(4)=0,\end{math} and \begin{math}\kappa(5)=1\end{math}. On the other hand, it is straightforward to check that
\begin{math}D(P)\end{math} is a connected graph for all \begin{math}n\geq 5\end{math}. From these facts and the definition of \begin{math}\kappa(D(P))\end{math} it follows that Theorem~\ref{thm:main}
 holds for each \begin{math} n\in \{3, \: 4, \: 5\}\end{math}. Thus we may assume that \begin{math}n\geq 6\end{math} and, according to Remark~\ref{lem:main}, all we need to show is that
 \begin{math}\eta(P;a,b)\geq \kappa(n)\end{math}.

Our proof of Theorem~\ref{thm:main} is mostly constructive, and the main steps are the following. First we prove that if \begin{math}a\end{math} and \begin{math}b\end{math} cross each other and have their four leaves in \begin{math}\overline{P}\end{math} (namely, Case 1), then there is a collection of pairwise internally disjoint
\begin{math}a-b\end{math} paths with cardinality \begin{math}\eta(P;a,b)\geq \kappa(n)\end{math}. We remark that such a collection of paths will be constructed in several ways,
depending on the values of  \begin{math}p, \: q, \: r\end{math}, and \begin{math}s\end{math}. Then, we observe that many of the paths constructed in Case 1 remain well defined and useful even if \begin{math}a\end{math} and \begin{math}b\end{math} do not satisfy the conditions of Case 1. Finally, in each case distinct from Case 1, we take the useful paths given in Case 1 and complete the required  collection of \begin{math}a-b\end{math} paths in certain way (which depends on the specific case).
\vskip 0.2cm

\textsc{ Case 1.} \textbf{Suppose that \begin{math}a\end{math} and \begin{math}b\end{math} cross at \begin{math}o\end{math} and have their four leaves in \begin{math}\overline{P}\end{math}.}

By performing a suitable rotation of all the points of \begin{math}P\end{math} around \begin{math}o\end{math}, if necessary, we may assume that
\begin{math}|Y^+|\geq \max\{|Y^-|, \: |X^-|, \: |X^+|\}\end{math}. Additionally, by reflecting \begin{math}P\end{math} along the \begin{math}y\end{math}-axis, if necessary, we also can assume that \begin{math}|X^+|\geq |X^-|\end{math}.

Our approach in this case is to give an explicit collection of pairwise internally disjoint
\begin{math}a-b\end{math} paths with cardinality \begin{math}\eta(P;a,b)\geq \kappa(n)\end{math}. Roughly speaking, we start by showing that \begin{math}agb\end{math} defines an \begin{math}a-b\end{math} path in \begin{math}D(P)\end{math} of length two for each \begin{math}g\in M\end{math}. These will be the \begin{math}a-b\end{math} paths provided by \begin{math}M\end{math}. Then, in each arising subcase, we will construct (explicitly) a bijective function with domain \begin{math}A\end{math} or almost all \begin{math}A\end{math} and codomain \begin{math}B\end{math}, such that if \begin{math}d\in A\end{math} and \begin{math}h\in B\end{math} are matched by that bijection, then \begin{math}adhb\end{math} defines an \begin{math}a-b\end{math} path in \begin{math}D(P)\end{math} of length \begin{math}3\end{math}. As we will see, the \begin{math}a-b\end{math} paths provided by that bijection together with those provided by \begin{math}M\end{math} give all (or almost all) the required collection.
\begin{remark}\label{rem:order}
Hence \begin{math} p, \: q, \: r,\end{math} and \begin{math}s\end{math} are integers such that \begin{math} s\geq r\geq 0, \: q\geq p\geq 1,\end{math} and  \begin{math} n=p+q+s+r+2\end{math}.
\end{remark}
Let \begin{math}x^-_1\end{math} and \begin{math}x^+_{q+1}\end{math} (respectively, \begin{math}x^+_1\end{math} and \begin{math}x^-_{p+1}\end{math}) be the leaves of \begin{math}a\end{math} (respectively, \begin{math}b\end{math}).
By supposition, \begin{math}x^-_1, \: x^+_{q+1}, \: x^+_1,\end{math} and \begin{math}x^-_{p+1}\end{math} are in \begin{math}\overline{P}\end{math}. Without loss of gene\-rality,
 we may assume that they are placed as in Figure~\ref{fig:VI}. We now label the rest of points of \begin{math}P\end{math}
 in radial order around \begin{math}o\end{math} as follows. If \begin{math}Y^-\end{math} (respectively, \begin{math}X^-, \: Y^+, \: X^+\end{math}) is nonempty, then we let \begin{math}Y^-=\{y^-_1,\ldots,y^-_r\}\end{math}
(respectively, \begin{math} X^-=\{x^-_2,\ldots,x^-_p\}, \:  Y^+=\{y^+_1,\ldots,y^+_s\}, \:  X^+=\{x^+_q, \ldots ,x^+_2\}\end{math}), where the order listed of the points in each set
corresponds to their clockwise order around \begin{math}o\end{math}. For convenience, we define \begin{math} y^-_0:=x^+_1, \: y^-_{r+1}:=x^-_1, \: y^+_0:=x^-_{p+1},\end{math} and \begin{math}y^+_{s+1}:=x^+_{q+1}\end{math}, as depicted in Figure~\ref{fig:VI}.
\begin{figure}
	\centering
	\includegraphics[width=0.52\textwidth]{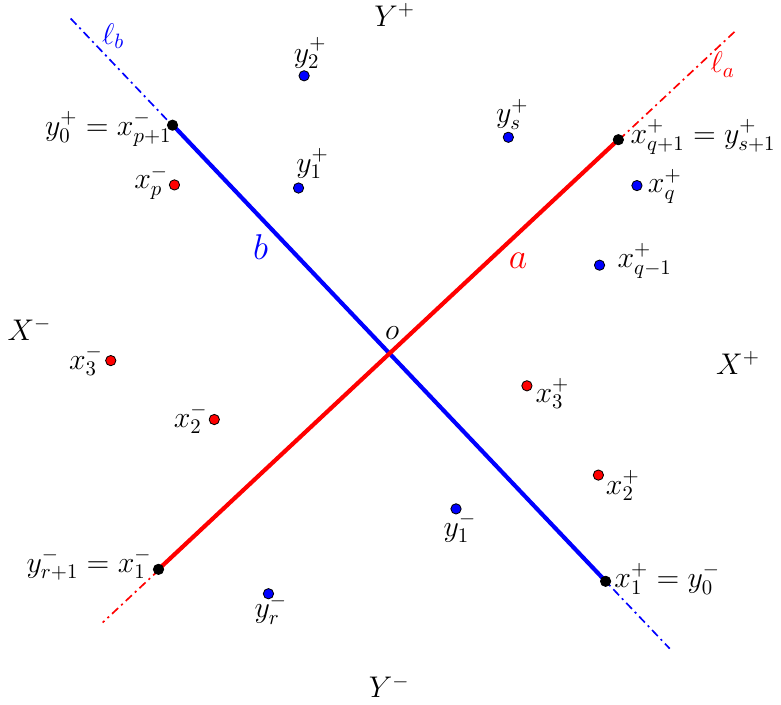}
	\caption{\small The point \begin{math}o=(0,0)\end{math} is an inner point of both \begin{math}a=x^-_1x^+_{q+1}\end{math} and \begin{math}b=x^-_{p+1}x^+_1\end{math}, and
	\begin{math}x^-_1, \: x^+_{q+1}, \: x^+_1,\end{math} and \begin{math}x^-_{p+1}\end{math} are in \begin{math}\overline{P}\end{math}. Here none of \begin{math}Y^-, \: Y^+, \: X^-,\end{math} and \begin{math}X^+\end{math} is empty.}
		\label{fig:VI}
\end{figure}
From Figure~\ref{fig:VI} it is easy to see that the segment \begin{math}uv\end{math} is a member of \begin{math}\wa\end{math}  if and only if  \begin{math}(u, v)\end{math} belongs to some of the following six subsets:
 \begin{math}X^+\times \{x^+_1\}, \: X^+\times Y^-, \: \{x^+_1\}\times Y^-, \: X^-\times \{x^-_{p+1}\}, \:  X^-\times Y^+,\end{math} and \begin{math}\{x^-_{p+1}\}\times Y^+\end{math}.
 Similarly, \begin{math}uv\end{math} is an element of \begin{math}\wb\end{math} if and only if \begin{math}(u, v)\end{math} belongs to some of the following six sets:  \begin{math}Y^+ \times \{x^+_{q+1}\}, \: Y^+\times X^+, \: \{x^+_{q+1}\}\times X^+, \: Y^-\times\{x^-_1\}, \: Y^-\times X^-,\end{math} and \begin{math}\{x^-_1\}\times X^-\end{math}. Then,
\begin{equation*}
\begin{split}
|\wa| = (q-1)+(q-1)r+r+(p-1)+(p-1)s+s= p(s+1)+q(r+1)-2,\\
|\wb| = s+s(q-1)+(q-1)+r+r(p-1)+(p-1) = p(r+1)+q(s+1)-2.\\
\end{split}
\end{equation*}
By the last equalities and Remark~\ref{rem:order} we get \begin{math}|A|\leq |B|\end{math}. And from the definition of \begin{math}\delta_3\end{math} it follows that:
\begin{equation}\label{Eq:paths3}
 \delta_3=|\wa|=p(s+1)+q(r+1)-2.
\end{equation}
We now proceed to produce the required \begin{math}\kappa(n)\end{math} \begin{math}a-b\end{math} paths.
\begin{proposition}\label{cl:Xlenght2}
Let  \begin{math} a, \: b, \: p, \: q, \: r \end{math} and \begin{math}s\end{math} be as above, and let \begin{math}\eta_2(a,b)\end{math} be the maximum number of pairwise internally disjoint \begin{math}a-b\end{math} paths of \begin{math}G\end{math} of length \begin{math}2\end{math}. Then
\begin{math}\eta_2(a,b)=\delta_2\end{math}, and \begin{displaymath}\delta_2=\binom{q-1}{2}+\binom{p-1}{2}+\binom{r}{2}+\binom{s}{2}.\end{displaymath}
 \end{proposition}
\begin{proof} From Proposition~\ref{cl:P0} we know that \begin{math}\eta_2(a,b)=\delta_2\end{math}.
On the other hand, since \begin{math}a\end{math} and \begin{math}b\end{math} cross and have four leaves in \begin{math}\overline{P}\end{math} it follows
that any segment of \begin{math}\w\end{math} must have both endpoints in exactly one of the following four sets:  \begin{math}X^+\end{math}, \begin{math}X^-\end{math}, \begin{math}Y^-,\end{math} or \begin{math}Y^+\end{math}.
Then \begin{math}\delta_2=|\w|= \binom{q-1}{2}+\binom{p-1}{2}+\binom{r}{2}+\binom{s}{2}\end{math}, as required.
\end{proof}

 For \begin{math}t\in \ZZ^+\end{math}, let \begin{math}X^+_t:=\{x^+_j~|~ x^+_j\in X^+ \mbox{ and } j\leq t\}\end{math} and  \begin{math}Y^+_t:=\{y^+_i~|~ y^+_i\in Y^+ \mbox{ and } i\leq t\}\end{math}.
Let \begin{math}\ii:=\{x^-_{p+1} y^+_{s}\}\cup\{x^+_1y^-_{r}~|~Y^-\neq\emptyset\}\cup \{x^+_1x^+_{q}~|~X^+\neq\emptyset\}\end{math}. We remark that
\begin{math}\ii=\{x^-_{p+1}y^+_{s}\}\end{math} whenever \begin{math}Y^-=\emptyset\end{math} and \begin{math}X^+=\emptyset\end{math}. Let \begin{math} A':=\wa\setminus \ii\end{math}.

 We now define in (\ref{Eq:function}) a mapping \begin{math}\psi\end{math} from \begin{math}A'\end{math} to \begin{math}\wb\end{math}. As we shall see later, \begin{math}\psi\end{math} will help us to construct a collection of \begin{math}|A'|\end{math} pairwise internally disjoint \begin{math}a-b\end{math} paths of \begin{math}G\setminus \w\end{math} of length \begin{math}3\end{math}.
\begin{equation}\label{Eq:function}
 \psi(uv):= \left\{ \begin{array}{lccc}
               \\y^-_1 x^-_{j}  &\mbox{ if } & (u,v)\in  A_1:=X^+_p\times \{x^+_1\}  \mbox{ and } uv = x^+_{j}x^+_1,\\
               \\y^-_{i+1} x^-_j &\mbox{ if } & (u,v)\in A_2:=X^+_p\times Y^-  \mbox{ and } uv=x^+_jy^-_i,\\
               \\ y^-_{i+1}x^-_1  &\mbox{ if } & (u,v)\in  A_3:=\{x^+_1\}  \times Y^-_{r-1} \mbox{ and } uv=x^+_1y^-_{i},\\
               \\ y^+_1x^+_{j} & \mbox{ if } & (u,v)\in A_4:=X^- \times \{x^-_{p+1}\}  \mbox{ and } uv=x^-_jx^-_{p+1},\\
               \\ y^+_{i+1}x^+_j   &\mbox{ if } & (u, v)\in  A_5:=X^-\times Y^+ \mbox{ and } uv=x^-_jy^+_i,\\
               \\ y^+_{i+1}x^+_{p+1}  &\mbox{ if } & (u,v)\in A_6:=\{x^-_{p+1}\}\times Y^+_{s-1} \mbox{ and } uv=x^-_{p+1} y^+_{i},\\
               \\y^+_{s-i+1} x^+_{j+1} &\mbox{ if } & (u, v)\in A_7:=(X^+\setminus X^+_p) \times Y^- \mbox{ and } uv=x^+_jy^-_i,\\
               \\x^+_{j+1}x^+_{q+1} &\mbox{ if } & (u,v)\in A_8:=(X^+_{q-1}\setminus X^+_p) \times \{x^+_1\} \mbox{ and }uv=x^+_{j}x^+_1.\\
                \end{array}
 \right.
\end{equation}
Note that if \begin{math}A_i\neq \emptyset\end{math}, then \begin{math}|\psi( A_i)|=|A_i|\end{math}. Indeed, from the definition of \begin{math}\psi\end{math} it is easy to see that,
\begin{equation}\label{Eq:imagen}
               \begin{array}{lccc}
               (1) \mbox{ } \psi( A_1) = \big\lbrace wz~|~(w,z)\in \{y^-_1\}\times X^- \big\rbrace\mbox{, and so } |\psi( A_1)|=p-1=|A_1|.\\
               (2) \mbox{ } \psi( A_2) = \big\lbrace wz~|~(w,z)\in \{y^-_2, \ldots , y^-_r, y^-_{r+1}\}\times X^- \big\rbrace \mbox{, and so } |\psi( A_2)|=r(p-1)=|A_2|.\\
               (3) \mbox{ } \psi( A_3) = \big\lbrace wz~|~(w,z)\in \{y^-_2, \ldots , y^-_r\}\times \{x^-_1\} \big\rbrace \mbox{, and so } |\psi( A_3)|=r-1=|A_3|.\\
               (4) \mbox{ } \psi( A_4) = \big\lbrace wz~|~(w,z)\in \{y^+_1\}\times X^+_p \big\rbrace \mbox{, and so } |\psi( A_4)|=p-1=|A_4|.\\
               (5) \mbox{ } \psi( A_5) = \big\lbrace wz~|~(w,z)\in \{y^+_2, \ldots , y^+_s, y^+_{s+1}\}\times X^+_p \big\rbrace \mbox{, and so } |\psi( A_5)|=s(p-1)=|A_5|.\\
               (6) \mbox{ } \psi( A_6) = \big\lbrace wz~|~(w,z)\in \{y^+_2, \ldots , y^+_s\}\times \{x^+_{p+1}\} \big\rbrace \mbox{, and so } |\psi( A_6)|=s-1=|A_6|.\\
               (7) \mbox{ } \psi( A_7) = \big\lbrace wz~|~(w,z)\in \{y^+_s,y^+_{s-1}, \ldots , y^+_{s-r+1}\}\times \{x^+_{p+2}, \ldots , x^+_{q+1} \} \big\rbrace, \\ \mbox{ and so } |\psi( A_7)|=r(q-p)=|A_7|.\\
               (8) \mbox{ } \psi( A_8) = \big\lbrace wz~|~(w,z)\in \{x^+_{p+2}, \ldots , x^+_{q}\}\times \{x^+_{q+1}\} \big\rbrace \mbox{, and so } |\psi( A_8)|=q-p-1=|A_8|.\\
                \end{array}
\end{equation}
\begin{proposition}\label{claim:function}
Let \begin{math}\psi : A' \longrightarrow \wb\end{math} be as above. Then
\begin{displaymath}\PP_1:=\big\lbrace a (uv) (\psi(uv)) b~|~uv\in A' \big\rbrace,\end{displaymath}
 is a collection of \begin{math}|A'|\end{math} pairwise internally disjoint \begin{math}a-b\end{math} paths of \begin{math}G\setminus \w\end{math} of length \begin{math}3\end{math}.
\end{proposition}
\begin{proof} Let \begin{math}uv\in A'\end{math}. From the definitions of \begin{math}uv\end{math} and \begin{math}\psi\end{math} it is easy to see that \begin{math}a (uv) (\psi(uv)) b\end{math}
satisfies each of the conditions (i)-(iii) of Observation~\ref{claim:conditions}. Then \begin{math}a (uv) (\psi(uv)) b\end{math} is an \begin{math}a-b\end{math} path of \begin{math}G\setminus \w\end{math} of length \begin{math}3\end{math}.

In order to show that the paths in \begin{math}\PP_1\end{math} are pairwise internally disjoint, it is enough to show that \begin{math}\psi\end{math} is injective.
Let \begin{math}u_1v_1\end{math} and \begin{math}u_2v_2\end{math} be distinct segments of \begin{math}A'\end{math}. A simple inspection of Equations~(\ref{Eq:imagen}) reveals that
\begin{math}\psi(A_{i_1})\cap \psi(A_{i_2})=\emptyset\end{math} whenever \begin{math}i_1\neq i_2\end{math}. This implies that  \begin{math}\psi(u_1v_1)\neq \psi(u_2v_2)\end{math} for
\begin{math}u_1v_1\in A_{i_1}\end{math} and \begin{math}u_2v_2\in A_{i_2},\end{math} with \begin{math}i_1\neq i_2\end{math}. Then we may assume that \begin{math}u_1v_1, \: u_2v_2\in A_{i_0}\end{math}
for some \begin{math}i_0\in \{1,2,\ldots ,8\}\end{math}. Again, from Equations~(\ref{Eq:imagen}) we know that \begin{math}|\psi(A_{i_0})|=|A_{i_0}|\end{math}. Since \begin{math}|A_{i_0}|\end{math} is finite, then the restriction of \begin{math}\psi\end{math}
 to \begin{math}A_{i_0}\end{math} is a bijection, and so \begin{math}\psi(u_1v_1)\neq \psi(u_2v_2)\end{math}, as required.
\end{proof}
\begin{proposition}\label{cl:Xlenght3}
Let  \begin{math} a, \: b, \: p, \:q, \:r,\end{math} and \begin{math}s\end{math} be as above, and let \begin{math}\eta_3(a,b)\end{math} be the maximum number of pairwise internally disjoint \begin{math}a-b\end{math} paths of \begin{math}G\setminus \w\end{math}  of length \begin{math}3\end{math}, then
\begin{equation}\label{Eq:eta3}
 \eta_3(a,b)= \left\{ \begin{array}{lcc}
                 \delta_3-1 & if &  q=1 \mbox{ and } r=0,\\
                   \\ \delta_3 & otherwise.  \\
             \end{array}
   \right.
\end{equation}
 \end{proposition}
\begin{proof} We recall that \begin{math}s\geq 1, \: s\geq r\geq 0\end{math}, and \begin{math}q\geq p\geq 1\end{math}. Let
 \begin{math}\psi, \:  A', \:  \ii, \:  A_i,\end{math} and \begin{math}\PP_1\end{math} be as above. From Proposition~\ref{claim:function} and Equation~(\ref{Eq:imagen})
 we know that the number of pairwise internally disjoint \begin{math}a-b\end{math} paths of \begin{math}G\setminus \w\end{math} of length \begin{math}3\end{math} provided by \begin{math}\PP_1\end{math} is equal to
 \begin{equation}\label{Eq:P}
\begin{split}
|\PP_1|=\sum_{i=1}^8 |A_i|=p(s+2)+(q-1)r-3+\max\{0,r-1\}+\max\{0,q-p-1\}.
\end{split}
\end{equation}
 From Equation~(\ref{Eq:paths3}) we know that \begin{math}\delta_3=p(s+1)+q(r+1)-2\end{math}.

\textbf{ (7.1)} Suppose that \begin{math}q=1\end{math} and \begin{math}r=0\end{math}. Since \begin{math}q\geq p\geq 1\end{math}, then \begin{math}p=1\end{math} and \begin{math}\delta_3=s\end{math}. From Equation~(\ref{Eq:P}) we have that \begin{math}|\PP_1|=s-1\end{math}, and so  \begin{math}\eta_3(a,b)\geq s-1\end{math}.

On the other hand, note that if \begin{math}adhb\end{math} is an \begin{math}a-b\end{math} path of \begin{math}G\setminus \w\end{math} of length 3, then there exist integers \begin{math}i, j\end{math}
 such that \begin{math}1\leq i<j\leq s\end{math}, and \begin{math} d=x^-_{p+1}y^+_i, \: h=x^+_{q+1}y^+_j\end{math}. From \begin{math}1\leq i<j\leq s\end{math} we can deduce that the number of such pairs \begin{math}(y^+_i,y^+_j)\end{math} is at most
 \begin{math}s-1\end{math}, and so \begin{math}\eta_3(a,b)=s-1\end{math}, as required. This proves the Case (7.1).

On the other hand, by Observation~\ref{claim:conditions} and Equation~(\ref{Eq:paths3}) we have that \begin{math}\eta_3(a,b)\leq \delta_3=p(s+1)+q(r+1)-2\end{math}. Then, for the rest of the cases, it is enough to exhibit a collection  of \begin{math}p(s+1)+q(r+1)-2\end{math} pairwise internally disjoint \begin{math}a-b\end{math} paths of \begin{math}G\setminus \w\end{math} of length \begin{math}3\end{math}.

\textbf{ (7.2)} Suppose that \begin{math}q>p\end{math} and \begin{math}r\geq 1\end{math}. From Equation~(\ref{Eq:P}) we have that
\begin{displaymath}|\PP_1|=p(s+2)+(q-1)r-3+(r-1)+(q-p-1)=p(s+1)+q(r+1)-5.\end{displaymath}
Let \begin{math}d_1:=x_{p+1}^-y^+_s, \: d_2:=x^+_1x^+_q, \: d_3:=x^+_1y^-_r, \: h_1:=x^+_2x^+_{q+1}, \: h_2:=y^+_1x^+_{q+1},\end{math} and \begin{math}h_3:=x^-_1y^-_1\end{math}, as depicted in Figure~\ref{fig:six}.
\begin{figure}
	\centering
	\includegraphics[width=0.55\textwidth]{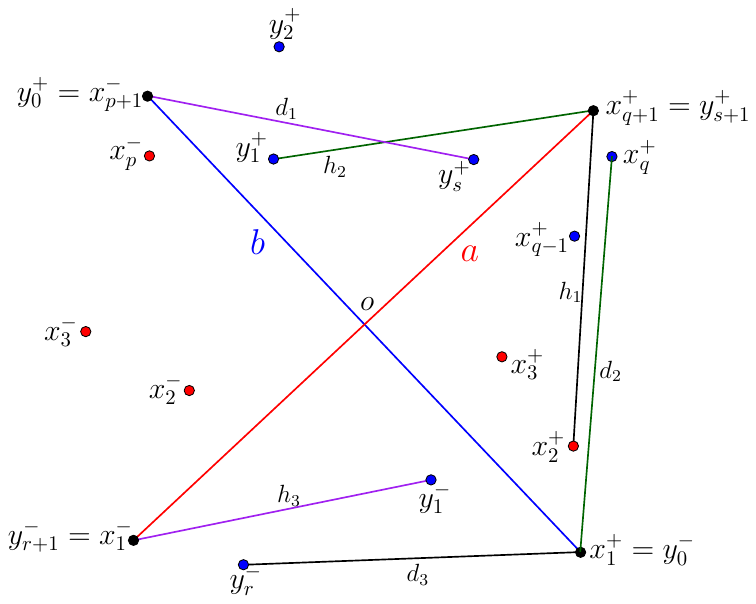}
	\caption{\small Here \begin{math}q>p\end{math} and \begin{math}r\geq 1\end{math}. Note that \begin{math}\ii=\{d_1, d_2, d_3\}\end{math} and \begin{math}\{h_1,h_2,h_3\}\in \wb\end{math}.}
		\label{fig:six}
\end{figure}
From \begin{math}q>p, \: r\geq 1\end{math}, and the definition of \begin{math}\PP_1\end{math} it is easy to check that none of \begin{math}d_1, \: d_2, \: d_3, \: h_1, \: h_2,\end{math} and \begin{math}h_3\end{math} belongs to any path of \begin{math}\PP_1\end{math}.
Similarly, note that \begin{math}d_i\in \wa, \: h_i\in \wb,\end{math} and \begin{math}d_i\cap h_{3-i+1}=\emptyset\end{math}, for \begin{math}i=1,2,3\end{math}. From these facts and
 Observation~\ref{claim:conditions} it follows that
\begin{displaymath}\PP_2:=\PP_1\cup \{a d_1 h_3 b, a d_2 h_2 b, a d_3 h_1 b\},\end{displaymath}
\noindent is the required collection.

\textbf{ (7.3)} Suppose that \begin{math}q>p\end{math} and \begin{math}r=0\end{math}. From Equation~(\ref{Eq:P}) we have that
 \begin{displaymath}|\PP_1|=p(s+2)+(q-1)r-3+(q-p-1)=p(s+1)+q(r+1)-4.\end{displaymath}
Let \begin{math}d_1, \: d_2, \: h_1,\end{math} and  \begin{math}h_2\end{math} as in Case (7.2). Again, note that \begin{math}d_i\in \wb, \: h_i\in \wa,\end{math} and \begin{math}d_i\cap h_{i}=\emptyset\end{math}, for \begin{math}i=1,2\end{math}. These and Observation~\ref{claim:conditions} imply that
\begin{displaymath}\PP_3:=\PP_1\cup \{a d_1 h_1 b, a d_2 h_2 b\},\end{displaymath}
\noindent is the required collection.

\textbf{ (7.4)} Suppose that \begin{math}q=p\end{math} and \begin{math}r=0\end{math}. By Case (7.1) we may assume that \begin{math}q=p\geq 2\end{math}. Then Equation~(\ref{Eq:P}) implies
that \begin{displaymath}|\PP_1|=p(s+2)+(q-1)r-3=p(s+1)+q(r+1)-3.\end{displaymath}
Since \begin{math}(d_1,h_1)\end{math} satisfies the conditions of  Observation~\ref{claim:conditions}, then \begin{displaymath}\PP_4:=\PP_1\cup \{a d_1 h_1 b\},\end{displaymath}
\noindent is the required collection.

\textbf{ (7.5)} Suppose that \begin{math}q=p\end{math} and \begin{math}r\geq 1\end{math}. Then Equation~(\ref{Eq:P}) implies
that \begin{displaymath}|\PP_1|=p(s+2)+(q-1)r-3+(r-1)=  p(s+1)+q(r+1)-4.\end{displaymath}
Let \begin{math}d_1, \: d_3, \: h_2,\end{math} and  \begin{math}h_3\end{math} as in Case (7.2).  Again, note that each pair \begin{math}(d_1,h_3)\end{math} and \begin{math}(d_3,h_2)\end{math} satisfies the conditions of
Observation~\ref{claim:conditions}, and so \begin{displaymath}\PP_5:=\PP_1\cup \{a d_1 h_3 b, a d_3 h_2 b\},\end{displaymath}
\noindent is the required collection.
\end{proof}
\begin{lemma}\label{lem:X}
If \begin{math}a\end{math} and \begin{math}b\end{math} cross at \begin{math}o\end{math} and have their four leaves in \begin{math}\overline{P}\end{math}, then \begin{math}\eta(P;a,b)=\delta(P;a,b).\end{math}
 \end{lemma}
\begin{proof} Trivially, \begin{math}\eta(P;a,b)\leq \delta(P;a,b)\end{math}. Then we need to show that \begin{math}\eta(P;a,b)\geq \delta(P;a,b)\end{math}.

 Let \begin{math} p, \: q, \: r, \: s, \: X^-, \: X^+, \: Y^-,\end{math} and \begin{math}Y^+\end{math} be as above.
For \begin{math} i\in \{1,\ldots ,5\}\end{math}, let \begin{math}\PP_0\end{math} and \begin{math}\PP_i\end{math} be as in the proofs of Propositions~\ref{cl:Xlenght2} and~\ref{cl:Xlenght3}, respectively. From the definition of
 \begin{math}\PP_i\end{math}, we know that no segment of \begin{math}\w\end{math} belongs to any path of \begin{math}\PP_i\end{math}, and hence \begin{math}\PP_0\cup \PP_i\end{math} is a collection of
\begin{math}|\PP_0|+|\PP_i|\end{math} pairwise internally disjoint \begin{math}a-b\end{math} paths of \begin{math}G\end{math}.

Let us first assume that the endpoints of \begin{math}a\end{math} and \begin{math}b\end{math} are not consecutive in \begin{math}\overline{P}\end{math}. Then at least one of \begin{math}q>1\end{math} or \begin{math}r>0\end{math} holds, and the corresponding
case in the proof of Proposition~\ref{cl:Xlenght3} is the Case (7.\begin{math}i\end{math}) for some \begin{math}i\in \{2,3,4,5\}\end{math}. In any of these four cases, we know from
Proposition~\ref{cl:Xlenght3} that \begin{math}|\PP_i|=\delta_3\end{math}, and so \begin{math}\PP_0\cup \PP_i\end{math} provides \begin{math}\delta_2+\delta_3=\delta(P;a,b)\end{math} pairwise internally disjoint \begin{math}a-b\end{math}
  paths of \begin{math}G\end{math}, and so \begin{math}\eta(P;a,b)\geq \delta(P;a,b)\end{math}, as required.

We now assume that the four endpoints of \begin{math}a\end{math} and \begin{math}b\end{math} are consecutive in \begin{math}\overline{P}\end{math}. Since \begin{math}Y^+\neq \emptyset\end{math}, then this case corresponds precisely
to the case in which \begin{math}X^-=\emptyset, \: Y^-=\emptyset,\end{math} and \begin{math}X^+=\emptyset\end{math} (or equivalently, \begin{math}q=p=1\end{math} and \begin{math}r=0\end{math}). Then \begin{math}a=x_1^-x_2^+\end{math} and \begin{math}b=x_1^+x_2^-\end{math}.
Moreover, note that \begin{math} x^-_2 x^-_1, \: x^-_1 x^+_1\end{math} and \begin{math} x^-_1 x^+_1, \: x^+_1 x^+_2\end{math} are pairs of consecutive segments in \begin{math}CH(P)\end{math}. Since \begin{math}n\geq 6\end{math}, then \begin{math}s= |Y^+| \geq 2\end{math}, and so
 \begin{math}y^+_1\end{math} and \begin{math}y^+_s\end{math} exist and are distinct. It is not hard to check that none of \begin{math} x^-_1 x^+_1, \: x^-_2 y^+_s,\end{math} or \begin{math}y^+_1x^+_2\end{math} belongs to any path of  \begin{math}\PP_0\cup \PP_1\end{math}.
 Since \begin{math}T^*:= a (x^-_2 y^+_s ) (x^-_1 x^+_1) (y^+_1x^+_2) b \end{math} is an \begin{math}a-b\end{math} path of \begin{math}D(P)\end{math} (of length 4), then from Propositions~\ref{cl:Xlenght2} and~\ref{cl:Xlenght3}
 we know that \begin{math}\PP_0\cup \PP_1 \cup \{ T^* \}\end{math} is a collection of  \begin{math}\delta_2+(\delta_3-1)+1=\delta(P;a,b)\end{math} pairwise internally disjoint \begin{math}a-b\end{math}
  paths of \begin{math}D(P)\end{math}, and hence \begin{math}\eta(P;a,b)\geq \delta(P;a,b)\end{math}, as required.
\end{proof}

The next result follows directly from Lemma~\ref{lem:X} and Proposition~\ref{prop:easy}, and concludes the proof of Case 1.
\begin{corollary}\label{cor:X}
 If \begin{math}a\end{math} and \begin{math}b\end{math} cross at \begin{math}o\end{math} and have their four leaves in \begin{math}\overline{P}\end{math}, then \begin{math}\eta(P;a,b)\geq \kappa(n).\end{math}
 \end{corollary}

We now emphasize some crucial properties of the collections of \begin{math}a-b\end{math} paths constructed above, which will be exploited in the next two cases. Note that all the \begin{math}\delta(P;a,b)\end{math} paths,
except
\begin{math}T^*= a (x^-_2 y^+_s ) (x^-_1 x^+_1) (y^+_1x^+_2) b\end{math}, provided by Lemma~\ref{lem:X} are contained in \begin{math}G\end{math}, and that only in the case when \begin{math}q=p=1,\end{math} and \begin{math}r=0\end{math}
was needed to use exactly a vertex not in \begin{math}G\end{math}, namely \begin{math}x^-_1 x^+_1\end{math}.

We will say that an \begin{math}a-b\end{math} path \begin{math}afgb\end{math} of length \begin{math}3\end{math} of \begin{math}D(P)\end{math} is \emph{ordered with respect to} \begin{math}o\end{math} if there exists  a straight line
\begin{math}\ell\end{math} passing through \begin{math}o\end{math} such that the vertices (segments) \begin{math}f\end{math} and \begin{math}g\end{math} lie on distinct sides of \begin{math}\ell\end{math}. Similarly, a collection of \begin{math}a-b\end{math} paths of \begin{math}D(P)\end{math} will be called \emph{ordered}
 if each of its paths of length \begin{math}3\end{math} is ordered with respect to \begin{math}o\end{math}. The following observation is easy to check and will be used in the next cases.
\begin{observation}\label{obs:ordered}
 Each path of length 3 constructed in Lemma~\ref{lem:X} is ordered with respect to \begin{math}o\end{math}, and hence each of the
 collections of \begin{math}\delta(P;a,b)\end{math} \begin{math}a-b\end{math} paths given in the proof of Lemma~\ref{lem:X} is ordered.
 \end{observation}

\textsc{ Case 2.} \textbf{Suppose that \begin{math}o\end{math} is a common endpoint of \begin{math}a\end{math} and \begin{math}b\end{math}, and that their two leaves are in \begin{math}\overline{P}\end{math}.}

By performing a suitable rotation of \begin{math}P\end{math} around \begin{math}o\end{math}, if necessary, we may assume that
the leaves of \begin{math}a\end{math} and \begin{math}b\end{math} have negative \begin{math}y\end{math}-coordinate.  Let \begin{math}x^-_1\end{math} and  \begin{math}x^+_1\end{math} be the leaves of \begin{math}a\end{math} and \begin{math}b\end{math}, respectively. Then \begin{math}a=ox^-_1\end{math} and  \begin{math}b=ox^+_1\end{math}.
Also, by reflecting \begin{math}P\end{math} along the \begin{math}y\end{math}-axis, if necessary, we can assume that \begin{math}p-1=|X^-(P)|\leq |X^+(P)|=q-1\end{math}. In order to use some of the facts showed in Case 1,
let us label the points of \begin{math}P\setminus \{o, x^-_1, x^+_1\}\end{math} as in that case. In particular, we let \begin{math}y_0^-:=x_1^+\end{math} and \begin{math}y_{r+1}^-:=x_1^-\end{math}, as depicted in Figure~\ref{fig:V}.
\begin{remark}\label{rem:order-v}
Hence \begin{math}p, \: q, \: r,\end{math} and \begin{math}s\end{math} are integers such that \begin{math}r, \: s\geq 0,\end{math} \begin{math}q\geq p\geq 1,\end{math} and \begin{math}n=p+q+s+r+1\end{math}.
\end{remark}
 Our strategy to prove Case 2 is as follows. The first part of the proof corresponds to the case in which \begin{math}s=0\end{math}, and the argument proceeds in the same way as in Case 1, but much shorter by virtue of several facts stated in Case 1. The last part of the proof is checked by induction on \begin{math}n\end{math}.
\begin{figure}
	\centering
	\includegraphics[width=0.45\textwidth]{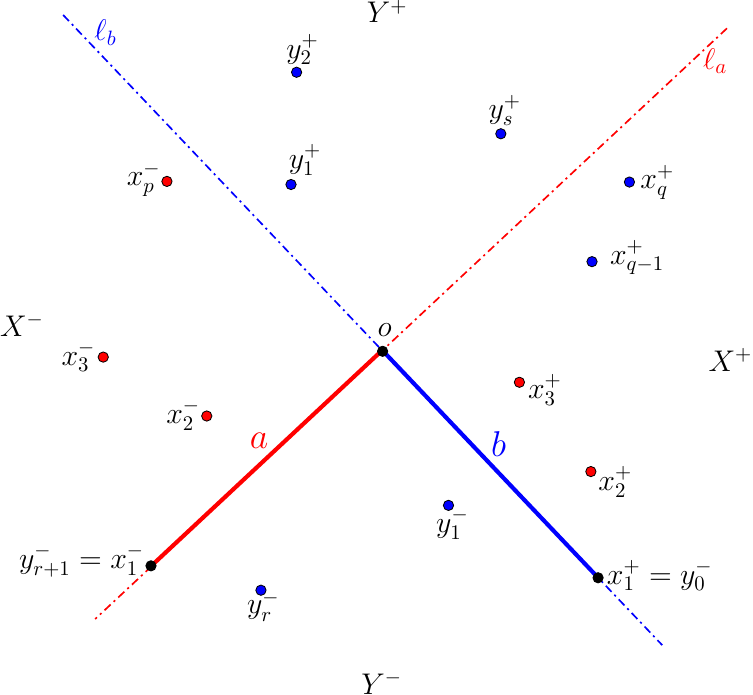}
	\caption{\small Here \begin{math}x^-_1, x^+_1\in \overline{P}\end{math}, \begin{math}a=ox^-_1,\end{math} \begin{math}b=ox^+_1\end{math}, and \begin{math}o=a\cap b\end{math} is a common endpoint of  \begin{math}a\end{math} and \begin{math}b\end{math}.}
\label{fig:V}
\end{figure}

The following is an analogue of Proposition~\ref{cl:Xlenght2}. The proof follows the same lines, and so we omit it.
\begin{proposition}\label{cl:Vlength2}
Let  \begin{math}a, b, p,q,r\end{math} and \begin{math}s\end{math} be as above, and let \begin{math}\eta_2(a,b)\end{math} be the maximum number of pairwise internally disjoint \begin{math}a-b\end{math} paths of \begin{math}G\end{math} of length \begin{math}2\end{math}.
Then \begin{math}\eta_2(a,b)=\delta_2=|\PP_0|\end{math}, and
\begin{displaymath}\delta_2\geq \binom{q-1}{2}+\binom{p-1}{2}+\binom{r}{2}+\binom{s}{2}+(p-1)s+(q-1)s,\end{displaymath}
where \begin{math}\PP_0\end{math} is as in Proposition~\ref{cl:P0}.
 \end{proposition}

\textsc{ Case 2.1}. \textbf{Suppose that \begin{math}s=0\end{math}.} Then \begin{math}\eta_2(a,b)\geq \binom{q-1}{2}+\binom{p-1}{2}+\binom{r}{2}\end{math} by Proposition~\ref{cl:Vlength2}, and so
\begin{math}D(P)\end{math} has a collection \begin{math}\PP_0\end{math} of pairwise internally disjoint \begin{math}a-b\end{math} paths of length 2 with \begin{math}|\PP_0|\geq \binom{q-1}{2}+\binom{p-1}{2}+\binom{r}{2}\end{math}.

As before, from Figure~\ref{fig:V} it follows that the segment \begin{math}uv\end{math} is a member of \begin{math}A\end{math} if and only if
 \begin{math}(u, v)\end{math} belongs to some of the following three sets:  \begin{math}X^+\times \{x^+_1\}, \{x^+_1\}\times Y^-, X^+\times Y^-\end{math}.
 Similarly, \begin{math}uv\end{math} is an element of \begin{math}B\end{math} if and only if  \begin{math}(u, v)\end{math} belongs to some of the following three subsets:
  \begin{math} Y^-\times\{x^-_1\}, \{x^-_1\}\times X^-,\end{math} and \begin{math}Y^-\times X^-\end{math}. Then,
\begin{equation*}
\begin{split}
|A|=(q-1+r)+r(q-1)=q(r+1)-1,\\
|B|=(p-1+r)+r(p-1)=p(r+1)-1.
\end{split}
\end{equation*}
By the last equalities and Remark~\ref{rem:order-v} we get \begin{math}|B|\leq |A|\end{math}. And from the definition of \begin{math}\delta_3\end{math} it follows that:

\begin{equation}\label{Eq:paths3V}
\delta_3=|B|=p(r+1)-1.
 \end{equation}
For \begin{math}t\in \ZZ^+\end{math}, let \begin{math}X^+_t, Y^+_t, A_1, A_2\end{math} and \begin{math}A_3\end{math} be as in Case 1. We remark that each of these sets is well defined in the context of current case. Then, for \begin{math} A'':=A_1\cup A_2\cup A_3\end{math}, we let \begin{math}\phi:A''\to B\end{math} denote the restriction of \begin{math}\psi\end{math} to \begin{math}A''\end{math}, where \begin{math}\psi\end{math} is as in Eq.~(\ref{Eq:function}). Then, according to~(\ref{Eq:imagen}), the following hold.
\begin{equation}\label{Eq:fV}
\begin{array}{lccc}
               (1) \mbox{ } \phi( A_1) = \big\lbrace wz~|~(w,z)\in \{y^-_1\}\times X^- \big\rbrace\mbox{, and so } |\phi( A_1)|=p-1=|A_1|.\\
               (2) \mbox{ } \phi( A_2) = \big\lbrace wz~|~(w,z)\in \{y^-_2, \ldots , y^-_r, y^-_{r+1}\}\times X^- \big\rbrace \mbox{, and so } |\phi( A_2)|=r(p-1)=|A_2|.\\
               (3) \mbox{ } \phi( A_3) = \big\lbrace wz~|~(w,z)\in \{y^-_2, \ldots , y^-_r\}\times \{x^-_1\} \big\rbrace \mbox{, and so } |\phi(A_3)|=r-1=|A_3|.\\
               \end{array}
\end{equation}
The next statement is an immediate consequence of Proposition~\ref{claim:function}.
\begin{proposition}\label{cl:Vfunction}
Let \begin{math}\phi : A'' \longrightarrow B\end{math} be as above. Then
\begin{displaymath}\PP_1:=\big\lbrace a (uv) (\phi(uv)) b~|~uv\in A'' \big\rbrace,\end{displaymath}
 is a collection of \begin{math}|A''|\end{math} pairwise internally disjoint \begin{math}a-b\end{math} paths of \begin{math}G\setminus M\end{math} of length \begin{math}3\end{math}.
\end{proposition}
Let \begin{math}\jj:=\{x^+_1x^+_{j}~|~j=p+1, \dots, q \}\end{math}.
Note that \begin{math}\jj\cap A''=\emptyset\end{math}, and also that \begin{math}\jj=\emptyset\end{math} for \begin{math}p=q\end{math}.
\begin{proposition}\label{cl:Vlength3}
Let  \begin{math}a, b, p,q,r\end{math} and \begin{math}s\end{math} be as above, and let \begin{math}\eta_3(a,b)\end{math} be the maximum number of pairwise internally disjoint \begin{math}a-b\end{math} paths of \begin{math}G\setminus M\end{math}  of length \begin{math}3\end{math}, then
\begin{equation}\label{Eq:eta3V}
\eta_3(a,b)\geq \left\{ \begin{array}{lcc}
                 \delta_3-1 & if &  p=q \mbox{ and } r>0,\\
                   \\ \delta_3 & otherwise.
             \end{array}
   \right.
\end{equation}
 \end{proposition}
\begin{proof} We recall that \begin{math}s=0\end{math}, \begin{math} r\geq 0\end{math}, and \begin{math} q\geq p\geq 1\end{math}. Let
 \begin{math}A_i, \phi, A'', \jj,\end{math} and \begin{math}\PP_1\end{math} be as above. From Proposition~\ref{cl:Vfunction} and Equation~(\ref{Eq:fV})
 we know that the number of pairwise internally disjoint \begin{math}a-b\end{math} paths of \begin{math}G\setminus M\end{math} of length \begin{math}3\end{math} provided by \begin{math}\PP_1\end{math} is equal to
 \begin{equation}\label{Eq:PV}
\begin{split}
|\PP_1|=\sum_{i=1}^3 |A_i|=p-1 + r(p-1)+\max\{0,r-1\}.
\end{split}
\end{equation}
 From Equation~(\ref{Eq:paths3V}) we recall that \begin{math}\delta_3=p(r+1)-1\end{math}.

Note that \begin{math}r= 0\end{math} implies \begin{math}\delta_3=p-1=|\PP_1|\end{math}, and so \begin{math}\eta_3(a,b)\geq |\PP_1|=\delta_3\end{math}, as required.
Thus we may assume that \begin{math}r\geq 1\end{math}, and so \begin{math}|\PP_1|=p(r+1)-2=\delta_3-1\end{math} by Eq.~(\ref{Eq:PV}).
It remains to show that if \begin{math}r\geq 1\end{math} and \begin{math}q>p\end{math}, then \begin{math}G\setminus M\end{math} has an \begin{math}a-b\end{math} path of length 3 that is independent of those in \begin{math}\PP_1\end{math}.

Suppose first that \begin{math}q>p\end{math}, and let \begin{math} d:=x^+_1x^+_{p+1}\in \jj\end{math} and \begin{math}h:=x^-_1y^-_1\end{math}. From \begin{math}q>p \geq 1 \end{math} and \begin{math}r\geq 1\end{math}, and the definition of \begin{math}\PP_1\end{math} it is easy to check that none
 of \begin{math}d\end{math} and \begin{math}h\end{math} belongs to any path of \begin{math}\PP_1.\end{math} Similarly, note that \begin{math}d\in A, h\in B,\end{math} and \begin{math} d\cap h=\emptyset\end{math}. From these facts and
 Observation~\ref{claim:conditions} it follows that \begin{math}adhb\end{math} is the required path.
\end{proof}
\begin{lemma}\label{lem:s=0}
If  \begin{math}a=ox^-_1\end{math}, \begin{math}b=ox^+_1\end{math} and \begin{math}x^-_1, x^+_1\in \overline{P}\end{math} and \begin{math}s=0\end{math}, then \begin{math}\eta(P;a,b)\geq\kappa(n).\end{math}
 \end{lemma}
\begin{proof}  Let \begin{math}p, \: q, \: r, \: s, \PP_0\end{math} and \begin{math}\PP_1\end{math} be as above. We recall that \begin{math}\delta(P;a,b)=\delta_2+\delta_3\end{math}.
As \begin{math}\delta(P;a,b)\geq \kappa(n)\end{math} by Proposition~\ref{prop:easy}, it is enough to show that \begin{math}\eta(P;a,b)\geq\delta_2+\delta_3\end{math}.

Suppose first that \begin{math}q>p\end{math} or \begin{math}r=0\end{math}. Then Propositions~\ref{cl:Vlength2} and \ref{cl:Vlength3} imply \begin{math}\eta_2(a,b)= \delta_2\end{math} and  \begin{math}\eta_3(a,b)\geq \delta_3\end{math}, respectively.
Therefore, \begin{math}\eta(P;a,b)\geq \delta_2+\delta_3\end{math}, as desired.

Suppose now that \begin{math}q=p=1\end{math} and \begin{math}r>0\end{math}. Then, \begin{math}r=n-3\end{math} because \begin{math}s=0\end{math}. Note that the assertion trivially holds for \begin{math}n=6\end{math}. Indeed,
 note that \begin{math}\kappa(6)=2\end{math} and \begin{math}\eta(P;a,b)\geq \delta_2=3\end{math} by Proposition~\ref{cl:Vlength2}. Thus, we can assume that \begin{math}n\geq 7\end{math}. From
 Propositions~\ref{cl:Vlength2} and \ref{cl:Vlength3} we know that \begin{math}\eta_2(a,b)\geq \binom{n-3}{2}\end{math} and \begin{math}\eta_3(a,b)\geq \delta_3-1=r-1=n-4\geq 3\end{math}, respectively.
 Then, \begin{math}\eta(a,b) \geq \binom{n-3}{2}+(n-4) \geq \binom{n-3}{2} + \binom{3}{2} \geq \kappa(n)\end{math}.

Suppose finally that \begin{math}q=p\geq 2\end{math} and \begin{math}r>0\end{math}. Since none of \begin{math}d:=x^+_1y^-_{r}\end{math}, \begin{math}h:=x^-_1y^-_1\end{math}, and \begin{math}e:=ox^+_q\end{math} belongs to any path of \begin{math}\PP_1\end{math},
then the \begin{math}a-b\end{math} path \begin{math}adehb\end{math} together with those \begin{math}\delta_2+\delta_3-1\end{math} provided
by Propositions~\ref{cl:Vlength2} and \ref{cl:Vlength3} imply \begin{math}\eta(a,b)\geq \delta_2+\delta_3\end{math}, as required.
\end{proof}

It is not hard to see that each \begin{math}a-b\end{math} path of length 3 described in Proposition~\ref{cl:Vlength3} is ordered with respect to \begin{math}o\end{math}, and so we can conclude
that Case 2.1 holds.

\textsc{ Case 2.2}. \textbf{Suppose that \begin{math}s\geq 1\end{math}.} We need to show that if \begin{math}s\geq 1\end{math}, then \begin{math}D(P)\end{math} has an ordered collection \begin{math}\TT\end{math} of \begin{math}\kappa(n)\end{math} pairwise internally disjoint \begin{math}a-b\end{math} paths. We proceed by induction on \begin{math}n\end{math}. As base case we take \begin{math}n=4\end{math}, for which there is nothing to prove because \begin{math}\kappa(4)=0\end{math}.
Thus, we can assume that \begin{math}n\geq 5\end{math} and that the statement holds for all \begin{math}m\in \{4,\ldots , n-1\}\end{math}. Since \begin{math}s\geq 1\end{math}, then \begin{math}Y^+(P)\end{math} has at least one point, say \begin{math}u\end{math}.

Suppose first that \begin{math}r=0\end{math}. By the induction hypothesis, we can assume that \begin{math}D(P\setminus\{u\})\end{math} has an ordered collection \begin{math}\TT'\end{math} of at least \begin{math}\kappa(n-1)\end{math} pairwise internally disjoint \begin{math}a-b\end{math} paths. Let \begin{math}{\mathcal Y}_u\end{math} be the set of straight segments with an endpoint at \begin{math}u\end{math} and the other in
\begin{math}X^-(P)\cup (Y^+(P)\setminus \{u\}) \cup X^+(P)\end{math}. Clearly, such segments are not vertices of
\begin{math}D(P\setminus \{u\})\end{math}. Moreover, since \begin{math}u\in Y^+(P)\end{math}, then each \begin{math} h\in {\mathcal Y}_u\end{math} is disjoint from \begin{math}a\end{math} and \begin{math}b\end{math}, and hence \begin{math}ahb\end{math} is an \begin{math}a-b\end{math} path of length 2 of \begin{math}D(P)\end{math}. The
\begin{math}\kappa(n-1)\end{math} \begin{math}a-b\end{math} paths of \begin{math}\TT'\end{math}, together with the \begin{math}(p-1)+(s-1)+(q-1)\end{math} paths provided by the elements of \begin{math}{\mathcal Y}_u\end{math}, yield an ordered collection \begin{math}\TT\end{math} of at least
\begin{math}\kappa(n-1)+p+q+s-3\end{math} pairwise internally disjoint \begin{math}a-b\end{math} paths of \begin{math}D(P)\end{math}.

On the other hand, from the definition of \begin{math}\kappa(n)\end{math} it is easy to see that
\begin{equation}\label{Eq:kappa}
 \kappa(n)-\kappa(n-1)= \left\{ \begin{array}{lcc}
                 \frac{n-3}{2} & if & n\mbox{ is odd,}\\
                 \frac{n-4}{2} & if & n\mbox{ is even.}\\
             \end{array}
   \right.
\end{equation}

From Remark~\ref{rem:order-v} we know that \begin{math}p+q+s+r=n-1\end{math}. This equality and \begin{math}r=0\end{math} imply \begin{math}p+q+s-3=n-4\end{math}. From
 Eq.~(\ref{Eq:kappa}) it is easy to see that \begin{math}\kappa(n-1)+(n-4)\geq \kappa(n)\end{math}, as claimed.

Suppose now that \begin{math}r\geq 1\end{math}. Then \begin{math}Y^-(P)\end{math} has at least one point, say \begin{math}v\end{math}. Let \begin{math}u\in Y^+(P)\end{math} and \begin{math}{\mathcal Y}_u\end{math} be as above.
By the induction hypothesis, we can assume that \begin{math}D(P\setminus\{u,v\})\end{math} has an ordered collection \begin{math}\TT'\end{math} of at least \begin{math}\kappa(n-2)\end{math} pairwise internally disjoint \begin{math}a-b\end{math} paths.
Let \begin{math}{\mathcal Y}_v\end{math} be the set of straight segments with an endpoint at \begin{math}v\end{math} and the other in \begin{math}Y^-(P)\setminus \{v\}\end{math}. Note that if
\begin{math}d:=x^+_1v, h:=vx^-_1,\end{math} and \begin{math}e:=ou\end{math}, then \begin{math}T:=adehb\end{math} is an \begin{math}a-b\end{math} path of \begin{math}D(P)\end{math}. Moreover, since no segment of
\begin{math}{\mathcal Y}_u\cup {\mathcal Y}_v\cup \{d,h,e\}\end{math} is a vertex of \begin{math}D(P\setminus\{u,v\})\end{math},
 then each \begin{math}g\in {\mathcal Y}_u\cup {\mathcal Y}_v\end{math} is disjoint from \begin{math}a\end{math} and \begin{math}b\end{math}, and hence \begin{math}agb\end{math} is an \begin{math}a-b\end{math} path of length 2 of \begin{math}D(P)\end{math}. The
\begin{math}\kappa(n-2)\end{math} \begin{math}a-b\end{math} paths of \begin{math}\TT'\end{math}, together with \begin{math}T\end{math} and the \begin{math}n-5\end{math} paths provided by the elements of \begin{math}{\mathcal Y}_u\cup {\mathcal Y}_v\end{math}, yield an ordered collection
\begin{math}\TT\end{math} of at least  \begin{math}\kappa(n-2)+(n-5)+1\end{math} pairwise internally disjoint \begin{math}a-b\end{math} paths of \begin{math}D(P)\end{math}. Again, from  Eq.~(\ref{Eq:kappa}) it is not hard to see that \begin{math}\kappa(n-2)+n-4\geq \kappa(n)\end{math}, as required.

In summary, we have proved the following lemma.
\begin{lemma}\label{lem:V}
If \begin{math}o\end{math} is a common endpoint of \begin{math}a\end{math} and \begin{math}b\end{math}, and their leaves are in \begin{math}\overline{P}\end{math}, then \begin{math}D(P)\end{math} has an ordered collection \begin{math}\QQ\end{math} of
pairwise internally disjoint \begin{math}a-b\end{math} paths with \begin{math}|\QQ|\geq \kappa(n)\end{math}.
 \end{lemma}
\textsc{ Case 3.} \textbf{Suppose that some leaf of \begin{math}\{a,b\}\end{math} does not belong to \begin{math}\overline{P}\end{math}.}
Our strategy for proving this case is as follows. First we enlarge \begin{math}a\end{math} and \begin{math}b\end{math} in such a way that the resulting objects \begin{math}P', \: a',\end{math} and \begin{math}b'\end{math} lie on some of the previous cases. Once we have the \begin{math}\kappa(n)\end{math} \begin{math}a'-b'\end{math} paths of \begin{math}D(P')\end{math} provided by Lemma~\ref{lem:X} or Lemma~\ref{lem:V},
we proceed to show that there exists a one-to-one mapping between such \begin{math}a'-b'\end{math} paths of \begin{math}D(P')\end{math} and certain subset of \begin{math}a-b\end{math} pairwise internally
disjoint paths of \begin{math}D(P)\end{math}. We formalize these ideas as follows.

 Let \begin{math}L\end{math} be the set of leaves of \begin{math}\{a,b\}\end{math}. Then \begin{math}|L|\in \{2,4\}\end{math}. We recall that \begin{math}o=(0,0)\end{math} is the intersection point between the segments \begin{math}a\end{math} and \begin{math}b\end{math}, and so \begin{math}o\notin L\end{math}.
 As before, by rotating \begin{math}P\end{math} around \begin{math}o\end{math} if necessary, we can assume that \begin{math}L\end{math} has exactly two leaves, say \begin{math}u\in a\end{math} and  \begin{math}v\in b\end{math},
with negative \begin{math}y\end{math}-coordinate. Without loss of generality we assume that \begin{math}u\end{math} has negative \begin{math}x\end{math}-coordinate and that \begin{math}v\end{math} has positive \begin{math}x\end{math}-coordinate, as depicted in Figure~\ref{fig:short}.

For \begin{math}x\in L\end{math}, let \begin{math}\ell_x\end{math} be the ray starting at \begin{math}o\end{math} and passing through \begin{math}x\end{math}. As \begin{math}P\end{math} is finite,
 then there exists a circumference \begin{math}O\subset \reals^2\end{math} centered at the origin \begin{math}o\end{math}, which contains \begin{math}P\end{math} in its interior.
Let us define \begin{math}\gamma(x):=\ell_x\cap O\end{math} for \begin{math}x\in L\end{math},  and \begin{math}\gamma(x):=x\end{math} for \begin{math}x\in P\setminus L\end{math}. Clearly,
we can choose \begin{math}O\end{math} so that the resulting \begin{math}n\end{math} point set \begin{math}P':=\{\gamma(x)~|~x\in P\}\end{math} is in general position. Then \begin{math}\gamma\end{math} is a bijection from \begin{math}P\end{math} to \begin{math}P'\end{math}. Moreover, it is not hard to see that if for \begin{math}xy\in \pp\end{math},
we let \begin{math}\gamma(xy):=\gamma(x)\gamma(y)\end{math}, then this ``extension" of \begin{math}\gamma\end{math} defines a bijection from \begin{math}\pp\end{math} to \begin{math}\pp'\end{math}. If \begin{math}h\end{math} is a vertex (segment) of \begin{math}D(P)\end{math}, then \begin{math}\gamma(h)\end{math} denote its corresponding vertex (segment) in \begin{math}D(P')\end{math}. It follows from the definition of \begin{math}\gamma\end{math} that \begin{math}a':=\gamma(a)\end{math} (respectively, \begin{math}b':=\gamma(b)\end{math}) contains \begin{math}a\end{math} (respectively, \begin{math}b\end{math}) as subsegment, as depicted in Figure~\ref{fig:short}.

Note that all the leaves of \begin{math}\{a', b'\}\end{math} are in \begin{math}\overline{P'}\end{math}. Let \begin{math}\TT'\end{math} be the ordered collection of \begin{math}a'-b'\end{math} pairwise internally
disjoint  paths of \begin{math}D(P')\end{math} constructed following the procedure described in the proof of Lemma~\ref{lem:X}
(respectively, Lemma~\ref{lem:V}) if \begin{math}|L|=4\end{math} (respectively, \begin{math}|L|=2\end{math}). Then \begin{math}\TT'\end{math} has at most one path of length \begin{math}4\end{math}, and the remaining paths
have length \begin{math}2\end{math} or \begin{math}3\end{math}. We also recall that \begin{math}|\TT'|\geq \kappa(n)\end{math}.

Let \begin{math}C'\end{math} be the set of vertices that are in \begin{math}D(P')\setminus D(P)\end{math}, i.e., the set of segments in \begin{math}{\mathcal P}'\end{math} that have at least one endpoint in \begin{math}\gamma(L)\end{math}.
Let \begin{math}T'\end{math} be a path of \begin{math}\TT'\end{math}. We shall use \begin{math}(T')\end{math} to denote the subpath obtained from \begin{math}T'\end{math} after deleting its endvertices (namely \begin{math}a'\end{math} and \begin{math}b'\end{math}).
Note that if \begin{math}(T')\end{math} has no vertices of \begin{math}C'\end{math}, then each vertex in \begin{math}(T')\end{math} is also a vertex of \begin{math}D(P)\end{math}, and so \begin{math}a(T')b\end{math} defines an \begin{math}a-b\end{math} path of \begin{math}D(P)\end{math}, since \begin{math}a\end{math} and \begin{math}b\end{math} are contained
in \begin{math}a'\end{math} and \begin{math}b'\end{math}, respectively. Then \begin{math}{\mathbb T}_0:=\{a(T')b~|~T'\in {\mathbb T}' \mbox{ and } T'\cap C'=\emptyset \}\end{math} is an ordered collection of pairwise internally disjoint \begin{math}a-b\end{math} paths of \begin{math}D(P)\end{math}.
Moreover, note that each path of length \begin{math}2\end{math} of \begin{math}\TT'\end{math} contributes to \begin{math}\TT_0\end{math}.
\vskip 1.55cm
We now focus on the collection \begin{math}\TT'_1\end{math} formed by the paths of \begin{math}\TT'\end{math} that intersect the set \begin{math}C'\end{math}. Thus, if \begin{math}T'\in \TT'_1\end{math}, then \begin{math}T'\cap C'\neq \emptyset\end{math} and \begin{math}T'\end{math} has length \begin{math}3\end{math} or \begin{math}4\end{math}.
Since \begin{math}\TT'\end{math} is ordered, we know that for each path \begin{math}T':=a'f'g'b'\end{math} of length 3 of \begin{math}\TT'\end{math}, there exists a line \begin{math}\ell_{f'g'}\end{math} which passes through \begin{math}o\end{math} and separates
\begin{math}f'\end{math} from \begin{math}g'\end{math}. It is no hard to see that the segments \begin{math}\gamma^{-1}(f')\end{math} and \begin{math}\gamma^{-1}(g')\end{math} in \begin{math}D(P)\end{math} remain separated by \begin{math}\ell_{f'g'}\end{math},
and hence \begin{math}a\gamma^{-1}(f')\gamma^{-1}(g')b\end{math} defines an \begin{math}a-b\end{math} path of \begin{math}D(P)\end{math}. Similarly, note that if \begin{math}\TT'\end{math} contains the path \begin{math}T^*\end{math} of length \begin{math}4\end{math} described in the proof of Lemma~\ref{lem:X} (respectively, Lemma~\ref{lem:V}), then
\begin{math}\gamma^{-1}(T^*)\end{math} defines an \begin{math}a-b\end{math} path of length \begin{math}4\end{math} in \begin{math}D(P)\end{math}. Since \begin{math}\gamma\end{math} is a bijection between \begin{math}\pp\end{math} and \begin{math}\pp'\end{math}, then
the \begin{math}|\TT'_1|\end{math} paths of \begin{math}D(P)\end{math} produced by \begin{math}\gamma^{-1}(\TT'_1)\end{math} are pairwise internally disjoint. These paths together with those in \begin{math}{\mathbb T}_0\end{math}
provides the required \begin{math}|\TT'|\geq  \kappa(n)\end{math} paths of \begin{math}D(P)\end{math}. This concludes the proof of  Case 3, and hence the proof of Theorem~\ref{thm:main}. \hfill \begin{math}\square\end{math}

\begin{figure}
	\centering
	\includegraphics[width=0.7\textwidth]{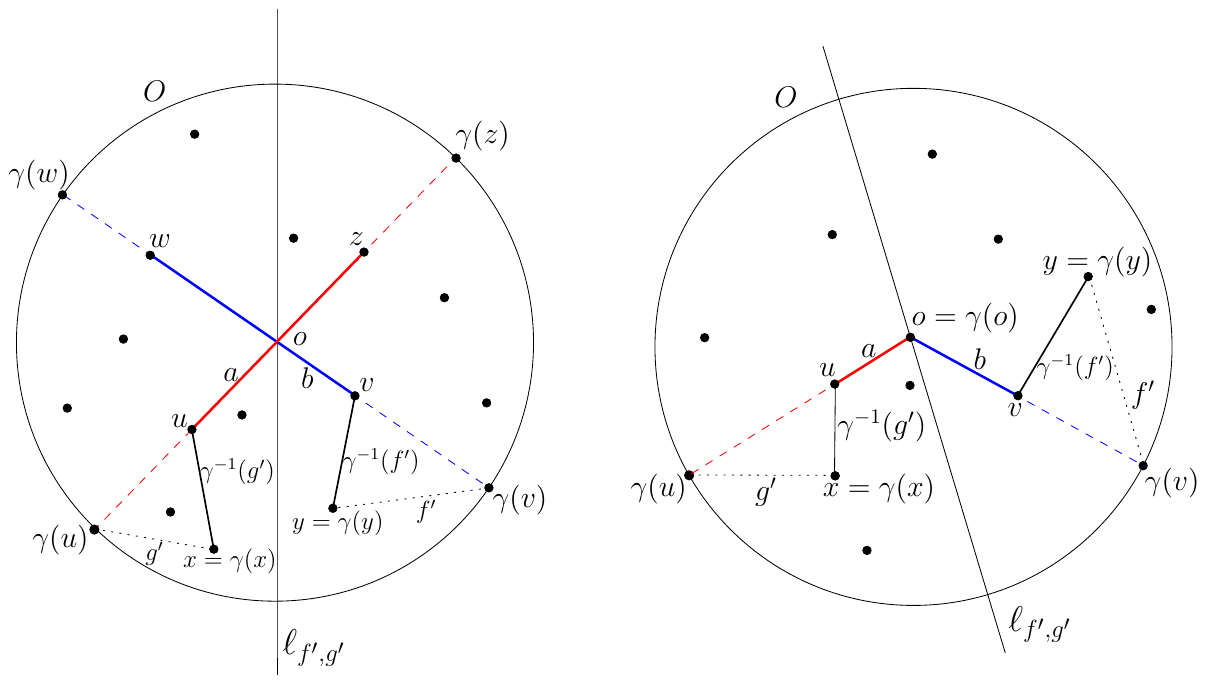}
	\caption{\small Two possibilities for \begin{math}P, \:  a,\end{math} and \begin{math}b\end{math}. In both cases \begin{math}O\end{math} is a circumference centered at \begin{math}o=(0,0)\end{math} and contains all the points of \begin{math}P\end{math}. Note that on the left case \begin{math}o\notin P\end{math}, but on the right we have \begin{math}o\in P\end{math}. In both instances \begin{math}\overline{P}\end{math} does not contain any leaf of
	\begin{math}\{a, b\}\end{math}. The set of leaves \begin{math}L\end{math} of \begin{math}\{a,b\}\end{math} on the left is \begin{math}\{u,v,w,z\}\end{math}, while for the set on the right we have that \begin{math}L=\{u,v\}\end{math}.}
		\label{fig:short}
\end{figure}
\section{Concluding remarks}\label{sec:conclutions}
A trivial upper bound for the connectivity \begin{math}\kappa(H)\end{math} of a graph \begin{math}H\end{math} is its minimum degree \begin{math}\delta(H)\end{math}. As we have observed in Proposition~\ref{prop:easy}, if \begin{math}P\end{math} is a set of \begin{math}n\geq 3\end{math} points in general position in the plane, then its disjointness graph of segments \begin{math}D(P)\end{math} has minimum degree \begin{math}\delta(D(P))\geq \kappa(n):=\binom{\lfloor\frac{n-2}{2}\rfloor}{2}+\binom{\lceil\frac{n-2}{2}\rceil}{2}.\end{math}

From Corollary~\ref{cor:main} it follows that if \begin{math}\delta(D(P))=\kappa(n)\end{math}, then \begin{math}\kappa(D(P))=\delta(D(P))\end{math}, and hence Theorem~\ref{thm:main} is best possible for such point sets. We remark that not only the points in convex position satisfy the hypothesis of Corollary~\ref{cor:main}, but also any point set \begin{math}P\end{math} containing two points in \begin{math}\overline{P}\end{math} such that the line spanned by them separates \begin{math}P\end{math} into two sets of sizes as equal as possible.  We believe that it would be interesting to determine a the exact value of
  \begin{math}\delta(D(P))-\kappa(D(P))\end{math} for the case in which \begin{math}\delta(D(P))>\kappa(n)\end{math}.

Finally, we recall that the basic idea of the proof of Theorem~\ref{thm:main} is to construct, for each pair of nonadjacent vertices \begin{math}a\end{math} and \begin{math}b\end{math} in
\begin{math}D(P)\end{math}, a collection with at least \begin{math}\kappa(n)\end{math} pairwise internally disjoint \begin{math}a-b\end{math} paths. Surprisingly, we were able to form all these collections with
paths of length at most \begin{math}4\end{math}, which implies that the diameter of \begin{math}D(P)\end{math} is between \begin{math}2\end{math} and \begin{math}4\end{math} whenever \begin{math}n\geq 5\end{math}.

\acknowledgements{The authors would like to thank the two anonymous referees for their constructive comments, which helped us to improve the manuscript.}

\nocite{*}
\bibliographystyle{abbrvnat}
\bibliography{sample-dmtcs}
\label{sec:biblio}

\end{document}